\documentclass{amsart}
\usepackage{mathrsfs}

\newtheorem{theorem}{Theorem}[section]
\newtheorem{lemma}[theorem]{Lemma}
\newtheorem{proposition}[theorem]{Proposition}
\newtheorem{corollary}[theorem]{Corollary}

\theoremstyle{definition}

\theoremstyle{remark}

\numberwithin{equation}{section}

\begin{document}

\title{Uniform bounds for expressions involving modified Bessel functions}

\author{Robert E$.$ Gaunt}
\address{Department of Statistics, University of Oxford, 24-29 St$.$ Giles', OX1 3LB, Oxford, UK}
\email{gaunt@stats.ox.ac.uk}


\date{February 2016}


\begin{abstract}In this paper, we obtain uniform bounds for a number of expressions that involve integrals of modified Bessel functions.  These uniform bounds are motivated by the need to bound such expressions in the study of variance-gamma and product normal approximations via Stein's method. 
\end{abstract}

\maketitle

\section{Introduction}

In developing Stein's method for variance-gamma and product normal distributions, Gaunt \cite{gaunt vg, gaunt elec, gaunt pn} required uniform bounds on the first four derivatives of the function
\begin{align*}f(x)&=-\frac{\mathrm{e}^{-\beta x} K_{\nu}(|x|)}{|x|^{\nu}} \int_0^x \mathrm{e}^{\beta y} |y|^{\nu} I_{\nu}(|y|) \tilde{h}(y) \,\mathrm{d}y \\
&\quad -\frac{\mathrm{e}^{-\beta x} I_{\nu}(|x|)}{|x|^{\nu}} \int_x^{\infty} \mathrm{e}^{\beta y} |y|^{\nu} K_{\nu}(|y|)\tilde{h}(y)\,\mathrm{d}y, \quad x\in\mathbb{R},
\end{align*}
where 
\begin{equation*}\tilde{h}(y)=h(y)-\int_{-\infty}^{\infty}\frac{(1-\beta^2)^{\nu+\frac{1}{2}}}{\sqrt{\pi}\Gamma(\nu+\frac{1}{2})2^{\nu}}\mathrm{e}^{\beta t} |t|^{\nu} K_{\nu}(|t|)h(t)\,\mathrm{d}t,
\end{equation*}
and $\nu>-\frac{1}{2}$, $-1<\beta<1$, and $h:\mathbb{R}\rightarrow\mathbb{R}$ is three times differentiable with bounded derivatives.  To achieve uniform bounds on these derivatives, we require bounds for a number of terms involving integrals of the modified Bessel functions $I_{\nu}(x)$ and $K_{\nu}(x)$.  In this paper, we establish uniform bounds for some of these terms.  

Before presenting the expressions that we obtain bounds for, we introduce the following notation for the the repeated integral of the function $\mathrm{e}^{\beta x}x^{\nu}I_{\nu}(x)$: 
\begin{equation} \label{super1} I_{(\nu,\beta,0)}(x)=\mathrm{e}^{\beta x}x^{\nu}I_{\nu}(x), \qquad I_{(\nu,\beta,n+1)}(x)=\int_0^xI_{(\nu,\beta,n)}(y)\,\mathrm{d}y,\quad n =0,1,2,3,\ldots.
\end{equation}

In this paper, we shall focus on bounding the expressions of the type
\begin{equation}\label{mlhwsb1} \frac{\mathrm{e}^{-\beta x}I_{\nu}(x)}{x^{\nu}} \int_x^{\infty} \mathrm{e}^{\beta t}t^{\nu} K_{\nu}(t)\,\mathrm{d}t 
\end{equation}
and
\begin{equation}\label{mlhwsb2} I_{(\nu,\beta,n)}(x)  \frac{\mathrm{e}^{-\beta x}K_{\nu+n}(x)}{x^{\nu}}, 
\end{equation}
where $n\geq 1$.  We shall bound (\ref{mlhwsb1}) for all $\nu>-\frac{1}{2}$ and $-1<\beta<1$ (see Theorem \ref{vole}), but will only bound (\ref{mlhwsb2}) for the case $n=1$ (see Theorem \ref{appd1}).  We shall then move on to consider the case $\beta=0$.  For this case, we obtain improved bounds for (\ref{mlhwsb1}) (see Theorem \ref{ml7min}) and are able to bound (\ref{mlhwsb2}) for all $n\geq 1$ (see Theorem \ref{ivy}).

This paper is organised as follows.  In Section 2, we state a number of results from the literature that we will make repeated use of.  We present formulas and inequalities for modified Bessel functions and their integrals that we make use in this paper.  In Section 3, we obtain uniform bounds for the expressions involving modified Bessel functions that have been presented in this section.  In the Appendix, we state a number of elementary properties of modified Bessel functions that are used throughout this paper.

\section{Ancillary results}

In this section, we state a number of results concerning modified Bessel functions that we will use in this paper.  
We begin by stating the following simple inequalities for integrals of modified Bessel functions, which can be found in Gaunt \cite{gaunt inequality}.
\begin{lemma} \label{tiger} The following inequalities hold for all $x>0$,
\begin{eqnarray}\label{intIi} \int_0^x t^{\nu} I_{\nu +n}(t)\,\mathrm{d}t &<& \frac{2(\nu + n +1)}{2\nu +n +1} x^{\nu} I_{\nu + n+1}(x), \quad \nu > -\tfrac{1}{2}, \: n\geq 0, \\
\label{intIii}I_{(\nu,0,n)}(x) &<& \bigg\{\prod_{k=1}^n \frac{2\nu + 2k}{2\nu + k}\bigg\} x^{\nu} I_{\nu + n}(x), \quad \nu \geq 0, \: n=1,2,3,\ldots,
\end{eqnarray} 
where $I_{(\nu,0,n)}(x)$ is defined as in (\ref{super1}).
\end{lemma}

\begin{lemma}\label{tiger1}Let $-1\leq \beta <1$, then, for all $x>0$, the following inequalities hold
\begin{eqnarray}\label{Kint11}\int_x^{\infty} t^{\nu} K_{\nu}(t)\,\mathrm{d}t &<& x^{\nu} K_{\nu +1}(x), \quad \nu \in \mathbb{R}, \\
 \label{KBinteqi} \int_x^{\infty}\mathrm{e}^{\beta t}t^{\nu}K_{\nu}(t)\,\mathrm{d}t &<& \frac{1}{1-|\beta|}\mathrm{e}^{\beta x}x^{\nu}K_{\nu}(x), \quad \nu < \tfrac{1}{2}, \\
 \label{gggh} \int_x^{\infty} t^{\nu} K_{\nu}(t)\,\mathrm{d}t &<& \frac{ \sqrt{\pi} \Gamma(\nu +\frac{1}{2})}{\Gamma(\nu)} x^{\nu} K_{\nu}(x), \quad \nu \geq \tfrac{1}{2}, \\
 \label{neo mice} \int_x^{\infty} \mathrm{e}^{\beta t} t^{\nu} K_{\nu}(t)\,\mathrm{d}t &<& \frac{ 2\sqrt{\pi} \Gamma(\nu +\frac{1}{2})}{(1-\beta^2)^{\nu+\frac{1}{2}}\Gamma(\nu)} \mathrm{e}^{\beta x}x^{\nu} K_{\nu}(x), \quad \nu \geq \tfrac{1}{2}.
\end{eqnarray}
\end{lemma}
The following proposition was proved for $\nu>-\frac{1}{2}$ in Baricz \cite{baricz} and then extended to $\nu>-1$ in Baricz \cite{baricz2}.  The result has a simple corollary (stated below), which we will make repeated use of.   
\begin{proposition}\label{IKmon}
For positive real argument and $\nu>-1$, the product  $K_{\nu}(x) I_{\nu}(x)$ is a strictly monotone decreasing function of $x$.
\end{proposition}
\begin{corollary}\label{IKineq}Let $\nu>0$, then for all $x \ge 0$, then the following inequality holds
\begin{equation} \label{mice} 0<K_{\nu}(x) I_{\nu}(x) \leq \frac{1}{2\nu}.\end{equation}
\end{corollary}
\begin{proof}The proof follows from Proposition \ref{IKmon} and the asymptotic expansions (\ref{Itend0}), (\ref{Ktend0}), (\ref{roots}) and (\ref{Ktendinfinity}) for modified Bessel functions in the limit $x$ tends to $0$ and $\infty$.
\end{proof}
We end this section by noting one final lemma, that is used in the proof of Theorem \ref{appd1}.

\begin{lemma}Fix $\nu\geq-\frac{1}{2}$.  Then, for all $x\geq0$, the function $xK_{\nu+1}(x)I_\nu(x)$ is strictly monotone decreasing in $x$, and satisfies the inequality
\begin{equation}\label{bdsjbc}\frac{1}{2}< xK_{\nu+1}(x)I_\nu(x)\leq1.
\end{equation}
\end{lemma}

\begin{proof}Firstly, we consider the case $\nu=-\frac{1}{2}$.  By (\ref{spheress}) and (\ref{sphere}) we have that $xK_{\frac{1}{2}}(x)I_{-\frac{1}{2}}(x)=\frac{1}{2}(1+\mathrm{e}^{-2x})$, which is strictly monotone decreasing and is bounded below and above by $\frac{1}{2}$ and $1$.  Now, we fix $\nu>-\frac{1}{2}$.  Applying the differentiation formulas (\ref{dddkkkk}) and (\ref{diffIii}) gives
\begin{align*}\frac{\mathrm{d}}{\mathrm{d}x}(xK_{\nu+1}(x)I_\nu(x))&=\frac{\mathrm{d}}{\mathrm{d}x}\bigg(x^{\nu+1}K_{\nu+1}(x)\cdot\frac{I_\nu(x)}{x^\nu}\bigg) \\
&=-x^{\nu+1}K_\nu(x)\cdot\frac{I_\nu(x)}{x^\nu}+x^{\nu+1}K_{\nu+1}(x)\cdot\frac{I_{\nu+1}(x)}{x^\nu} \\
&=x(K_{\nu+1}(x)I_{\nu+1}(x)-K_\nu(x)I_\nu(x)).
\end{align*}
Theorem 1 of Segura \cite{segura} states that, for $x>0$ and $\mu>-\frac{1}{2}$, the inequality $I_{\mu+1}(x)K_{\mu+1}(x)-I_{\mu}(x)K_{\mu}(x)<0$ holds.  Hence, $xK_{\nu+1}(x)I_\nu(x)$ is a decreasing function of $x$, and computing its limits as $x\downarrow0$ and $x\rightarrow\infty$ using the asymptotic formulas of Appendix A.5 yields inequality (\ref{bdsjbc}).
\end{proof}

\section{Uniform bounds for some expressions involving modified Bessel functions}

With our preliminary results stated, we are now able to present our bounds for the expressions involving modified Bessel functions that were presented in the introduction.  We begin by considering the general case $-1<\beta<1$ and $\nu>-\frac{1}{2}$ and then specialise to the case $\beta=0$. 

\begin{theorem}\label{appd1}Let $-1<\beta<1$ and $\nu>-\frac{1}{2}$.  Then, for all $x\geq0$,
\begin{equation*}\frac{\mathrm{e}^{-\beta x}K_{\nu+1}(x)}{x^{\nu}}\int_0^x\mathrm{e}^{\beta t}t^{\nu}I_{\nu}(x)\,\mathrm{d}t\leq\frac{2}{2\nu+1}.
\end{equation*}
\end{theorem}

\begin{proof} \label{noraz}  Note that
\[\frac{\mathrm{d}}{\mathrm{d}\beta}\bigg(\frac{\mathrm{e}^{-\beta x}K_{\nu+1}(x)}{x^{\nu}}\int_0^x\mathrm{e}^{\beta t}t^{\nu}I_{\nu}(x)\,\mathrm{d}t\bigg)=\frac{\mathrm{e}^{-\beta x}K_{\nu+1}(x)}{x^{\nu}}\int_0^x(t-x)\mathrm{e}^{\beta t}t^{\nu}I_{\nu}(x)\,\mathrm{d}t\leq0.\]
Therefore, for $-1<\beta<1$, we have
\begin{align*}\frac{\mathrm{e}^{-\beta x}K_{\nu+1}(x)}{x^{\nu}}\int_0^x\mathrm{e}^{\beta t}t^{\nu}I_{\nu}(x)\,\mathrm{d}t&\leq\frac{\mathrm{e}^{ x}K_{\nu+1}(x)}{x^{\nu}}\int_0^x\mathrm{e}^{- t}t^{\nu}I_{\nu}(x)\,\mathrm{d}t \\
&=\frac{1}{2\nu+1}xK_{\nu}(x)[I_{\nu}(x)+I_{\nu+1}(x)] \\
&\leq\frac{2}{2\nu+1}xK_{\nu+1}(x)I_{\nu}(x) \\
&\leq\frac{2}{2\nu+1},
\end{align*}
where we used the integral formula (\ref{pokemon}) to evaluate the integral, inequality (\ref{Imon}) to obtain the penultimate inequality and inequality (\ref{bdsjbc}) to obtain the final inequality.  
\end{proof}

\begin{theorem} \label{vole}Suppose $-1<\beta <1$ and $n=0,1,2,\ldots$. Then, for all $x\geq0$,
\begin{equation*}\left| \frac{\mathrm{e}^{-\beta x}I_{\nu}(x)}{x^{\nu}} \int_x^{\infty} \mathrm{e}^{\beta t}t^{\nu} K_{\nu}(t)\,\mathrm{d}t\right| < \frac{\sqrt{\pi} \Gamma(\nu +\frac{1}{2})}{(1-\beta^2)^{\nu+\frac{1}{2}}\Gamma(\nu+1)}, \quad \nu \geq \tfrac{1}{2}, 
\end{equation*}
and
\begin{equation*}\left| \frac{\mathrm{e}^{-\beta x}I_{\nu}(x)}{x^{\nu}} \int_x^{\infty} \mathrm{e}^{\beta t}t^{\nu} K_{\nu}(t)\,\mathrm{d}t\right|<\frac{(\mathrm{e}+1)2^\nu \Gamma(\nu+\frac{1}{2})}{1-|\beta|}, \quad |\nu|<\tfrac{1}{2}.
\end{equation*}
\end{theorem}

\begin{proof} (i) Suppose that $\nu \geq \frac{1}{2}$.  By inequality (\ref{neo mice}), we have 
\[\frac{\mathrm{e}^{-\beta x}I_{\nu}(x)}{x^{\nu}} \int_x^{\infty} \mathrm{e}^{\beta t} t^{\nu} K_{\nu}(t)\,\mathrm{d}t < \frac{ 2\sqrt{\pi} \Gamma(\nu +\frac{1}{2})}{(1-\beta^2)^{\nu+\frac{1}{2}}\Gamma(\nu)} I_{\nu}(x)K_{\nu}(x).\]
Using inequality (\ref{mice}) now gives
\begin{align*}\frac{\mathrm{e}^{-\beta x}I_{\nu}(x)}{x^{\nu}} \int_x^{\infty} \mathrm{e}^{\beta t} t^{\nu} K_{\nu}(t)\,\mathrm{d}t &<\frac{ 2\sqrt{\pi} \Gamma(\nu +\frac{1}{2})}{(1-\beta^2)^{\nu+\frac{1}{2}}\Gamma(\nu)}\cdot\frac{1}{2\nu} \\
&=\frac{ \sqrt{\pi} \Gamma(\nu +\frac{1}{2})}{(1-\beta^2)^{\nu+\frac{1}{2}}\Gamma(\nu+1)},
\end{align*}
as required.

(ii) Suppose now that $-\frac{1}{2}<\nu < \frac{1}{2}$.  We begin by proving that the bound holds in the region $x\geq \frac{1}{2}$.  By inequality  (\ref{KBinteqi}), we have 
\[\frac{\mathrm{e}^{-\beta x}I_{\nu}(x)}{x^{\nu}} \int_x^{\infty} \mathrm{e}^{\beta t} t^{\nu} K_{\nu}(t)\,\mathrm{d}t< \frac{1}{1-|\beta|}I_{\nu}(x)K_{\nu}(x).\]
By Proposition \ref{IKmon}, $I_{\nu}(x)K_{\nu}(x)$ is a monotone decreasing function of $x$ for $x>0$, and therefore we may bound this product, for $x\geq \frac{1}{2}$, by $I_{\nu}(\frac{1}{2})K_{\nu}(\frac{1}{2})$.  In fact, we may obtain a bound for all $0\leq \nu < \frac{1}{2}$ using (\ref{pidgeon}) and (\ref{bear}), which gives $I_{\nu}(\frac{1}{2})K_{\nu}(\frac{1}{2})<I_{-\frac{1}{2}}(\frac{1}{2})K_{\frac{1}{2}}(\frac{1}{2})=1+\mathrm{e}^{-1}$, where we used the formulas (\ref{spheress}) and (\ref{sphere}) for $I_{-\frac{1}{2}}(x)$ and $K_{\frac{1}{2}}(x)$, respectively, to obtain the equality.  Putting this together gives
\[ \frac{\mathrm{e}^{-\beta x}I_{\nu}(x)}{x^{\nu}} \int_x^{\infty} \mathrm{e}^{\beta t} t^{\nu} K_{\nu}(t)\,\mathrm{d}t <\frac{1+\mathrm{e}^{-1}}{1-|\beta|}, \quad x\geq \tfrac{1}{2}, \: |\nu| < \tfrac{1}{2}.\]
For $-\frac{1}{2}<\nu<\frac{1}{2}$, $\Gamma(\nu+\frac{1}{2})>\Gamma(1)=1$.  Therefore $1+\mathrm{e}^{-1}<(\mathrm{e}+1)2^{\nu}\Gamma(\nu+\frac{1}{2})$ for $-\frac{1}{2}<\nu<\frac{1}{2}$, and thus the bound holds in the region $x \geq \frac{1}{2}$.

We now verify that the bound holds in the region $0\leq x\leq \frac{1}{2}$.  From the series expansion (\ref{defI}) for $I_{\nu}(x)$, we can easily deduce that $x^{-\nu}I_{\nu}(x)$ is an increasing function of $x$.  The integral $\int\nolimits_x^{\infty}\mathrm{e}^{\beta t} t^{\nu} K_{\nu}(t)\,\mathrm{d}t$ is a decreasing function of $x$, and so we may bound the right-hand side of the previous display by 
\begin{equation} \label {whale}\frac{\mathrm{e}^{-\beta/2}I_{\nu}(\frac{1}{2})}{(\frac{1}{2})^{\nu}}\int_{-\infty}^{\infty} \mathrm{e}^{\beta t} |t|^{\nu} K_{\nu}(|t|)\,\mathrm{d}t =\frac{\sqrt{\pi}\mathrm{e}^{-\beta/2}I_{\nu}(\frac{1}{2})2^{\nu}\Gamma(\nu+\frac{1}{2})}{(1-\beta^2)^{\nu+\frac{1}{2}}},
\end{equation}
where the integral was evaluated using formula (\ref{pdfk}).  For $-1<\beta<1$ and $-\frac{1}{2}<\nu<\frac{1}{2}$ the following inequalities hold: $\mathrm{e}^{-\beta/2}<\mathrm{e}^{\frac{1}{2}}$, $I_{\nu}(\frac{1}{2})<I_{-\frac{1}{2}}(\frac{1}{2})=\pi^{-\frac{1}{2}}(\mathrm{e}^{\frac{1}{2}}+\mathrm{e}^{-\frac{1}{2}})$ and $(1-\beta^2)^{\nu+\frac{1}{2}}>1-|\beta|$.  With these bounds we may bound the right-hand side of (\ref{whale}), and thus obtain, for $0\leq x\leq \frac{1}{2}$ and $-\frac{1}{2}<\nu<\frac{1}{2}$,
\begin{equation*} \frac{2^n\mathrm{e}^{-\beta/2}I_{\nu}(\frac{1}{2})}{(\frac{1}{2})^{\nu}}\int_{-\infty}^{\infty} \mathrm{e}^{\beta t} |t|^{\nu} K_{\nu}(|t|)\,\mathrm{d}t <\frac{(\mathrm{e}+1)2^{\nu}\Gamma(\nu+\frac{1}{2})}{1-|\beta|}.
\end{equation*}
Combining this bound with the bound for $x\geq \frac{1}{2}$, and using that $\Gamma(\nu+\frac{1}{2})>1$ for $-\frac{1}{2}<\nu<\frac{1}{2}$, completes the proof of part (ii).
\end{proof}

We now specialise to the case $\beta=0$.  

\begin{theorem} \label{ivy} Let $I_{(\nu,0,n)}(x)$ be defined as per equation (\ref{super1}).  Suppose $\nu> -\frac{1}{2}$.  Then, for all $x\geq0$,
\begin{equation*} \left|I_{(\nu,0,n)}(x)  \frac{K_{\nu+n+1}(x)}{x^{\nu}}\right| \leq \frac{2^{n-1}}{2\nu +1}, \quad n\geq 1.
\end{equation*}
\end{theorem}

\begin{proof} Applying inequalities (\ref{intIii}) and (\ref{mice}) gives
\begin{align*}\left|I_{(\nu,0,n)}(x)  \frac{K_{\nu+n+1}(x)}{x^{\nu}}\right|
&< \bigg\{\prod_{k=1}^n \frac{2\nu + 2k}{2\nu + k}\bigg\}K_{\nu +n+1}(x) I_{\nu + n+1}(x) \\
&\leq \frac{1}{2(\nu +n +1)} \prod_{k=1}^n \frac{2\nu + 2k}{2\nu + k}. 
\end{align*}
We can simplify the bound given in the above display by noting
\[\frac{1}{2(\nu+n+1)}\prod_{k=1}^n \frac{2\nu + 2k}{2\nu + k}=\frac{1}{2\nu+1}\cdot \frac{2(\nu+n)}{2(\nu+n+1)}\prod_{k=1}^{n-1}\frac{2\nu+2k}{2\nu+k+1}< \frac{2^{n-1}}{2\nu+1},
\]
as $\frac{2\nu+2k}{2\nu+k+1}\leq \frac{2k-1}{k}<2$, for $\nu>-\frac{1}{2}$ and $k\geq 2$.  This completes the proof.
\end{proof}

\begin{theorem}\label{ml7min} Let $\nu>-\frac{1}{2}$. Then, for all $x\geq0$,
\begin{equation*}\frac{I_{\nu+1}(x)}{x^{\nu}}\int_x^{\infty}t^{\nu}K_{\nu}(y)\,\mathrm{d}t < \frac{1}{2(\nu+1)},
\end{equation*}
and
\begin{equation}\label{ml3min}\frac{I_{\nu}(x)}{x^{\nu}}\int_x^{\infty}t^{\nu}K_{\nu}(t)\,\mathrm{d}t\leq\frac{\sqrt{\pi}\Gamma(\nu+\frac{1}{2})}{2\Gamma(\nu+1)},
\end{equation}
with equality in the limit $x\downarrow0$.
\end{theorem}

\begin{proof}(i) Applying inequality (\ref{Kint11}) and then inequality (\ref{mice}) gives
\begin{equation*}\frac{I_{\nu+1}(x)}{x^{\nu}}\int_x^{\infty}t^{\nu}K_{\nu}(t)\,\mathrm{d}t < \frac{I_{\nu+1}(x)}{x^{\nu}}\cdot x^{\nu}K_{\nu+1}(x) \leq \frac{1}{2(\nu+1)}.
\end{equation*}

(ii) We begin by proving that $\frac{I_{\nu}(x)}{x^{\nu}}\int_x^{\infty}t^{\nu}K_{\nu}(t)\,\mathrm{d}t$ is a decreasing function.  By the differentiation formula (\ref{diffIii}) and inequality (\ref{Kint11}), we have, for $x\geq 0$,
\begin{align*}\frac{\mathrm{d}}{\mathrm{d}x}\bigg(\frac{I_{\nu}(x)}{x^{\nu}}\int_x^{\infty}t^{\nu}K_{\nu}(t)\,\mathrm{d}t\bigg)&=\frac{I_{\nu+1}(x)}{x^{\nu}}\int_x^{\infty}t^{\nu}K_{\nu}(t)\,\mathrm{d}t-I_{\nu}(x)K_{\nu}(x) \\
&\leq I_{\nu+1}(x)K_{\nu+1}(x)-I_{\nu}(x)K_{\nu}(x).
\end{align*}
Theorem 1 of Segura \cite{segura} states that for $x>0$ and $\mu>-\frac{1}{2}$ the inequality $I_{\mu+1}(x)K_{\mu+1}(x)-I_{\mu}(x)K_{\mu}(x)<0$ holds.  Hence, $\frac{I_{\nu}(x)}{x^{\nu}}\int_x^{\infty}t^{\nu}K_{\nu}(t)\,\mathrm{d}t$ is a decreasing function of $x$.

Using the asymptotic property (\ref{Itend0}) of modified Bessel functions $I_{\nu}(x)$ and that $\int_0^{\infty}t^{\nu}K_{\nu}(t)\,\mathrm{d}t=\sqrt{\pi}\Gamma(\nu+\frac{1}{2})2^{\nu-1}$ (see (\ref{pdfk})), we have
\[\lim_{x\rightarrow 0^+}\bigg(\frac{I_{\nu}(x)}{x^{\nu}}\int_x^{\infty}t^{\nu}K_{\nu}(t)\,\mathrm{d}t\bigg)=\frac{\sqrt{\pi}\Gamma(\nu+\frac{1}{2})}{2\Gamma(\nu+1)},\]
proving the result.
\end{proof}

We end this section by presenting some inequalities involving integrals of $I_0(x)$ and $K_0(x)$.  Bounds on these quantities are required in the development of Stein's method for product normal approximation; see Gaunt \cite{gaunt pn}.

\begin{theorem}For all $x\geq0$,
\begin{eqnarray*}\left|xI_{(0,0,n)}(x)  K_{n}(x)\right| &\leq& 2^{n-1}, \quad n=1,2,3,\ldots,\\
xI_{(0,0,1)}(x)  K_0(x) &<& 1,\\
xI_0(x)\int_x^{\infty}K_0(t)\,\mathrm{d}t&<& 0.615,\\
xI_1(x)\int_x^{\infty}K_0(t)\,\mathrm{d}t&\leq& \frac{1}{2}.
\end{eqnarray*}
\end{theorem}

\begin{proof}(i) From inequality (\ref{intIii}), we have
\[\left|xI_{(0,0,n)}(x)  K_n(x)\right|\leq 2^nxI_n(x)K_n(x).\]
From Theorem 4.1 of Hartman \cite{hartman}, we have that, for $\mu>\frac{1}{2}$, the function $xI_{\mu}(x)K_{\mu}(x)$ is an increasing function of $x$ (see also Baricz et al$.$ \cite{baricz3} for a number of upper bounds for the product $I_{\nu}(x)K_{\nu}(x)$).  On applying the asymptotic formulas (\ref{Ktendinfinity}) and (\ref{Ktendinfinity}) it follows that $\lim_{x\rightarrow \infty}xI_n(x)K_n(x)=\frac{1}{2}$, from which the desired inequality now follows.

(ii) This follows on applying the bound from part (i) and using the inequality $K_0(x)\leq K_1(x)$.

(iii) Applying inequality (\ref{KBinteqi}) gives
\[xI_0(x)\int_x^{\infty}K_0(t)\,\mathrm{d}t\leq xI_0(x)K_0(x).\]
Theorem 4.2 of Hartman \cite{hartman} states that there exists a constant $\tau$ such that the function $xI_{0}(x)K_{0}(x)$ is increasing on $(0,\tau)$ and decreasing on $(\tau,\infty)$.  The derivative of $xI_{0}(x)K_{0}(x)$ is given by $I_0(x)K_0(x)+x(I_1(x)K_0(x)-I_0(x)K_1(x))$, and is positive at $x=1$ and negative at $x=1.1$.  As $I_{0}(x)$ is an increasing function of $x$ and $K_{1}(x)$ is a decreasing function of $x$, it follows, on invoking the theorem of Hartman, that $xI_{0}(x)K_{0}(x)$ is bounded by $1.1I_0(1.1)K_1(1)=0.614\ldots<0.615$. 

(iv) Applying inequality (\ref{Kint11}) and arguing as we did in part (i) yields
\[xI_1(x)\int_x^{\infty}K_0(t)\,\mathrm{d}t\leq xI_1(x)K_1(x)\leq\frac{1}{2},\]
as required.
\end{proof}

\appendix

\section{Elementary properties of modified Bessel functions}

Here we list standard properties of modified Bessel functions that are used throughout this paper.  All these formulas can be found in Olver et al$.$ \cite{olver}, unless otherwise stated.  Inequalities (\ref{pidgeon}) and (\ref{Imon}) can be found in Jones \cite{jones} and N\r{a}sell \cite{nasell}; inequalities (\ref{bear}) and (\ref{cake}) can easily be deduced from (\ref{tellas}).  The integration formulas can be found in Gradshetyn and Ryzhik \cite{gradshetyn}.

\subsection{Basic properties}
The modified Bessel functions $I_{\nu}(x)$ and $K_{\nu}(x)$ are both regular functions of $x$.  For positive values of $x$ the functions $I_{\nu}(x)$ and $K_{\nu}(x)$ are positive for $\nu>-1$ and all $\nu\in\mathbb{R}$, respectively.

\subsection{Series expansions}
\begin{align}\label{defI}I_{\nu} (x) = \sum_{k=0}^{\infty} \frac{1}{\Gamma(\nu +k+1) k!} \left( \frac{x}{2} \right)^{\nu +2k}.
\end{align}

\subsection{Spherical Bessel functions}
\begin{eqnarray} \label{spheress} I_{-\frac{1}{2}}(x)&=&\sqrt{\frac{2}{\pi x}}\cosh(x), \\
 \label{sphere} K_{\frac{1}{2}}(x)&=&K_{-\frac{1}{2}}(x)=\sqrt{\frac{\pi}{2x}} \mathrm{e}^{-x}.
\end{eqnarray}

\subsection{Integral representations}
\begin{equation} \label{tellas} K_{\nu}(x)=\int_0^{\infty}\mathrm{e}^{-x\cosh(t)}\cosh(\nu t)dt, \quad x>0.
\end{equation}

\subsection{Asymptotic expansions}
\begin{eqnarray}\label{Itend0}I_{\nu} (x) &\sim& \frac{1}{\Gamma(\nu +1)} \left(\frac{x}{2}\right)^{\nu}, \quad x \downarrow 0, \\
\label{Ktend0}K_{\nu} (x) &\sim& \begin{cases} 2^{|\nu| -1} \Gamma (|\nu|) x^{-|\nu|}, & \quad x \downarrow 0, \: \nu \not= 0, \\
-\log x, & \quad x \downarrow 0, \: \nu = 0, \end{cases} \\
 \label{roots} I_{\nu} (x) &\sim& \frac{\mathrm{e}^x}{\sqrt{2\pi x}}, \quad x \rightarrow \infty,  \\
\label{Ktendinfinity} K_{\nu} (x) &\sim& \sqrt{\frac{\pi}{2x}} \mathrm{e}^{-x}, \quad x \rightarrow \infty.
\end{eqnarray}

\subsection{Inequalities}
Let $x > 0$, then following inequalities hold
\begin{eqnarray} \label{pidgeon} I_{\mu} (x) &<& I_{\nu } (x), \quad 0\leq \nu < \mu, \\
\label{Imon}I_{\nu} (x) &<& I_{\nu - 1} (x), \quad \nu \geq \tfrac{1}{2}, \\
 \label{bear} K_{\mu} (x) &>& K_{\nu } (x), \quad 0\leq \nu < \mu, \\
\label{cake}K_{\nu} (x) &\geq& K_{\nu - 1} (x), \quad \nu \geq \tfrac{1}{2}.  
\end{eqnarray}
We have equality in (\ref{cake}) if and only if $\nu=\frac{1}{2}$.

\subsection{Differentiation}
\begin{eqnarray}
\label{dddkkkk}\frac{\mathrm{d}}{\mathrm{d}x}(x^\nu K_\nu(x))&=&-x^\nu K_{\nu-1}(x), \\
\label{diffIii}\frac{\mathrm{d}}{\mathrm{d}x} \left(\frac{I_{\nu} (x)}{x^{\nu}} \right) &=& \frac{I_{\nu +1} (x)}{x^{\nu}}.
\end{eqnarray}

\subsection{Integration}
\begin{eqnarray} \label{pdfk} \int_{-\infty}^{\infty}\mathrm{e}^{\beta t} |t|^{\nu} K_{\nu}(|t|)\,\mathrm{d}t &=&\frac{\sqrt{\pi}\Gamma(\nu+\frac{1}{2})2^{\nu}}{(1-\beta^2)^{\nu+\frac{1}{2}}}, \quad \nu>-\tfrac{1}{2}, \: -1<\beta <1, \\
 \label{pokemon} \int_0^x\mathrm{e}^{-t}t^{\nu}I_{\nu}(t)\,\mathrm{d}t &=&\frac{\mathrm{e}^{-x}x^{\nu+1}}{2\nu+1}[I_{\nu}(x)+I_{\nu+1}(x)], \quad \nu>-\tfrac{1}{2}.
\end{eqnarray}

\section*{Acknowledgements}During the course of this research, the author was supported by an EPSRC DPhil Studentship and an EPSRC Doctoral Prize.  The author is currently supported by supported by EPSRC grant EP/K032402/1.  The author would like to thank Gesine Reinert for the valuable guidance she provided on this project.  Finally, the author would like to thank the referee for some useful comments.

\end{document}